\documentclass[12pt]{amsart}
\usepackage[left=1.5in, right=1.5in, top=1.5in, bottom=1.5in]
{geometry}
\newtheorem{thm}{Theorem}[section]
\newtheorem{cor}[thm]{Corollary}

\newtheorem{prop}[thm]{Proposition}
\theoremstyle{definition}

\newtheorem{rem}[thm]{Remark}
\usepackage{graphicx}
\usepackage{amssymb}
\usepackage{amsmath}
\usepackage{epstopdf}
\usepackage{hyperref}
\usepackage[colorinlistoftodos,prependcaption,textsize=tiny,
textwidth=0.75in, color=cyan]{todonotes}
\title[]{Minimizing the free energy}
\author{Emanuel Indrei and Aram Karakhanyan}
\address{Department of Mathematics\\
Purdue University\\
West Lafayette, Indiana \\
USA.}
\address{School of Mathematics\\ The University of Edinburgh\\ Peter
Tait Guthrie
Road\\ EH9 3FD Edinburgh, UK.
}
\thanks{ The first author was partially supported by Eliwise
Academy. The research of the second author was partially supported
by EPSRC grant
EP/S03157X/1 {\em Mean curvature measure of free boundary}.}
\begin{document}
\setcounter{page}{1}
\pagenumbering{arabic}
\maketitle
\begin{abstract}
We prove the sharp quantitative stability in the radial isotropic
Almgren problem. In
addition, we develop a theory for estimating the sharp modulus in
the context of
minimal assumptions on the surface tension and the potential and
obtain the sharp $
\epsilon^2$ in any dimension. Inter-alia, we also solve the problem
of calculating the
critical mass which was only a priori assumed to exist and which
breaks the mass
regime into two sets: the one where the energy is concave and the
one where it is
convex.
\end{abstract}
\section{Introduction}
Two main ingredients define the free energy of a set of finite
perimeter $E \subset
\mathbb{R}^n$ with reduced boundary $\partial^* E$:
$$
\mathcal{F}(E)=\int_{\partial^* E} f(\nu) d\mathcal{H}^{n-1}
$$
(the surface energy);
and,
$$
\mathcal{G}(E)=\int_E g(x)dx
$$
(the potential energy), where $g \ge 0$ is the potential, $g(0)=0$,
$g \in L_{loc}^\infty$, and $f$ is a surface tension, i.e. a convex
positively $1-$
homogeneous
$$f:\mathbb{R}^n\rightarrow [0,\infty).$$ \cite{Cryst}.
In this context, the free energy is
$$
\mathcal{E}(E)=\mathcal{F}(E)+\mathcal{G}(E).
$$
The thermodynamics minimizes the free energy under a mass constraint
($|E|=m$)
and generates a crystal. The underlying phenomenon was independently
discovered
by W. Gibbs in 1878 \cite{G} and P. Curie in 1885 \cite{Crist}.
Assuming the
gravitational effect is negligible, the solution is the convex set
$$
K= \bigcap_{v \in \mathbb{S}^{n-1}} \{x \in \mathbb{R}^n: \langle x,
v\rangle< f(v)\}
$$
(Wulff's theorem \cite{flashes, MR493671, MR1130601}). The longstanding
problem of investigating the convexity of minimizers when
$g$ is convex, known as Almgren's problem, has as its starting point Wulff's theorem. In a recent
paper, Indrei proved the convexity in two dimensions via assumptions
which are near optimal \cite{Cryst} (see in addition \cite{D, pFZ, Crystj}).
The only stability up to now appeared in Figalli's complete Fields Medal
Citation for $g=0$ with
an explicit modulus in the context of \cite{MR2672283}:
$w_m(\epsilon)=c(n)\epsilon^2 m^{\frac{n-1}{n}}$; in
\cite{MR2456887} for
$g=0$ and the isotropic case with a semi-explicit modulus; and, in
\cite{Cryst} for
$m$ small with a semi-explicit modulus and a locally bounded
potential.
A surprising technique enables us to solve the quantitative
stability problem in the
radial isotropic problem (Theorem \ref{@'}). The minimizers are
balls and a sharp
explicit modulus is illuminated: $w_{m,g}(\epsilon)=c(n, m, g)\epsilon^2$, with an
explicit $c(n,m, g)>0$ (cf. Remark \ref{zo}).
In addition, we develop a theory for estimating the sharp modulus in
the context of
minimal assumptions on $g$ and $f$ (Theorem \ref{@7z'}). One
interesting result is
that we obtain the sharp $\epsilon^2$ in any dimension.
Inter-alia we also solve the problem of calculating the critical
mass which breaks the
mass regime into two sets: the one where the energy is concave and
the one where
it is convex (Corollary \ref{@4'}). Last, we apply the results with
a new approach in $\mathbb{R}^3$ via the maximum principle combined with Simon's identity to obtain convexity in
more general environments.
\section{Quantitative stability in Almgren's free energy minimization}
\subsection{A theory for estimating the sharp modulus}
\begin{thm} \label{@'}
Suppose $f(\nu)=1$ if $\nu \in \mathbb{S}^{n-1}$, $g(x)=h(|x|)$, $h:
\mathbb{R}^+
\rightarrow \mathbb{R}^+$ is non-decreasing, $h(0)=0$. Let $m>0$, $|
B_a|=|E|=m$,
then\\
\noindent (i)
$$\mathcal{E}(E)-\mathcal{E}(B_a) \ge \int_{E \setminus B_a} [h(|
x|)-h(a)]dx+a_n
m\inf_{z} \Big(\frac{|E \Delta (B_a+z)|}{m}\Big)^2,
$$
where $a_n>0$ is a dimensional constant. \\
\noindent (ii)
$$
\hskip .2in \mathcal{E}(E)-\mathcal{E}(B_a) \ge \int_{E \setminus
B_a} \int_a^{|x|}
h'(t)dt dx+a_n m\inf_{z} \Big(\frac{|E \Delta (B_a+z)|}{m}\Big)^2,
$$
in addition, assuming $h$ is convex the subsequent inequalities are
valid:\\
\noindent (iii)
$$
\hskip .2in \mathcal{E}(E)-\mathcal{E}(B_a) \ge \partial_+h(a)
\int_{E \setminus B_a}
[|x|-a] dx+a_n m\inf_{z} \Big(\frac{|E \Delta (B_a+z)|}{m}\Big)^2,
$$
where $a_n>0$ and $\partial_+h(a)$ is any slope of a supporting line
at $a$;\\
\noindent (iv)
$$
\hskip .2in \alpha_1[[\partial_+h(a)]^{-1}[\mathcal{E}(E)-
\mathcal{E}(B_a)]]^{1/2} \ge |
E\Delta B_a|,
$$
where $\alpha_1=\alpha_1(n,a)>0$ is explicit.
\end{thm}
\begin{proof}
To start, assume $h'$ exists.
If $S \subset \mathbb{R}^+$, $|S|<\infty$, note that the
monotonicity yields
$$
\int_S h(x)dx \ge \int_0^{|S|} h(x)dx.
$$
Therefore if $S^-=S \cap \{x<0\}$, $S^+=S \cap \{x\ge 0\}$, note
that since $h$ is
even
$$
\int_S h \ge \int_{-|S^-|}^{|S^+|} h=\int_0^{|S^+|}h+\int_0^{|S^-|}
h.
$$
Now set
$$
H(y)=\int_0^y h(x)dx
$$
and consider
$$
H''(y)=h'(y) \ge 0.
$$
Therefore via convexity,
\begin{align*}
H(\frac{1}{2}(|S^-|+|S^+|))& \le \frac{1}{2}(H(|S^-|)+H(|S^+|))\\
&=\frac{1}{2}\Big(\int_0^{|S^+|}h+\int_0^{|S^-|}h\Big).
\end{align*}
This then implies in view that $h$ is assumed even,
$$
\int_{-\frac{1}{2}(|S^-|+|S^+|)}^{\frac{1}{2}(|S^-|+|S^+|)}h \le
\int_S h.
$$
Assume $\{w_1, w_2, \ldots \}$ are directions which generate via
Steiner
symmetrization with respect to the planes through the origin and
with normal $w_i$,
$E_{w_1}, E_{w_2}, \ldots$ and
$$
E_{w_i} \rightarrow B_a.
$$
Set
$$
E_1=\{(x_2, \ldots, x_n): (x_1, x_2, \ldots, x_n) \in E\}
$$
$$E^{x_2, x_3, \ldots, x_n}=\{x_1: (x_1, x_2, \ldots, x_n) \in E\}.$$
Note that one may let $\frac{x_1}{|x_1|}=w_1$;
\begin{align*}
\int_E g dx&=\int_{E_1} \Big( \int_{E^{x_2, x_3, \ldots, x_n}} g
dx_1\Big) dx_2\ldots
dx_n\\
&\ge \int_{E_1} \Big( \int_{-\frac{|E^{x_2, x_3, \ldots, x_n}|}{2}}
^{\frac{|E^{x_2, x_3,
\ldots, x_n}|}{2}} g dx_1\Big) dx_2\ldots dx_n\\
&=\int_{E_{w_1}} g dx.
\end{align*}
Moreover, because $g$ is radial, one may rotate the coordinate $x_1$
so that $
\frac{x_1}{|x_1|}=w_2$, and iterate the argument above:
$$
\int_{E_{w_1}} g dx \ge \int_{E_{w_2}} g dx;
$$
in particular,
$$
\int_{E} g dx \ge \int_{B_a} g dx.
$$
Since symmetrization never increases the perimeter, this then proves
the first part. In
order to remove the differentiability assumption on $h$, let $h_k
\rightarrow h$ a.e.
with $h_k$ differentiable and also bounded. \\
\noindent \emph{Case 1: $E$ bounded}\\
\noindent Observe that for $E \subset B_A$ for an $A>0$
$$\mathcal{E}_k(E):=\int_{\partial^* E} d\mathcal{H}^{n-1}+\int_E
g_k(x)dx \ge
\mathcal{E}_k(B_a);$$
in particular, dominated convergence implies
$$
\int_{B_a} g_k(x)dx \rightarrow \int_{B_a} g(x)dx
$$
$$
\int_{E} g_k(x)dx \rightarrow \int_{E} g(x)dx
$$
and this yields
$$
\mathcal{E}(E) \ge \mathcal{E}(B_a).
$$
\noindent \emph{Case 2: $E$ unbounded} \\
\noindent If $E$ is not bounded, let $E_T=E \cap B_T$ and note that
for $T$ large,
$|E_T| \approx m$, therefore
$$
\mathcal{E}(E_T)\ge \mathcal{E}(B_{a_T}),
$$
with $|E_T|=|B_{a_T}|$.
Moreover, since $E$ is of finite perimeter, thanks to the monotone
convergence
theorem
$$
\mathcal{H}^{n-1}(\partial^* E \setminus B_T) \rightarrow 0;
$$
this implies
$$
\mathcal{H}^{n-1}(\partial^* E \cap \partial B_T) \rightarrow 0;
$$
now via
$$\mathcal{H}^{n-1}(\partial^* E_T)= \mathcal{H}^{n-1}(\partial^* E
\cap \partial B_T)+
\mathcal{H}^{n-1}(\partial^* E \cap B_T)$$
and
$$
\mathcal{H}^{n-1}(\partial^* E \cap B_T) \rightarrow \mathcal{H}
^{n-1}(\partial^* E),
$$
the subsequent is true
$$\mathcal{H}^{n-1}(\partial^* E_T) \rightarrow \mathcal{H}^{n-1}
(\partial^* E).$$
In particular
$$
\mathcal{E}(E_T) \rightarrow \mathcal{E}(E),
$$
and via dominated convergence
$$
\mathcal{E}(B_{a_T}) \rightarrow \mathcal{E}(B_a).
$$
If $\mathcal{E}(E)= \mathcal{E}(B_a),$ observe that if $E$ is not
$B_a$, the above
then implies that the energy is decreased with a symmetrization,
which generates a
contradiction. Thus $\mathcal{E}(E)= \mathcal{E}(B_a)$ iff $E=B_a$.
Suppose now that
$$T: E\setminus B_a \rightarrow B_a \setminus E$$
denotes the Brenier map between $\mu = \chi_{E\setminus B_a} dx$ and
$
\nu=\chi_{B_a \setminus E} dx$ \cite{pFZ4}. Thanks to
$$
T_{\#} \mu=\nu,
$$
$$
\int_{E \setminus B_a} h(|T(x)|)dx=\int_{B_a \setminus E} h(|x|)dx;
$$
thus, because $h$ is non-decreasing
\begin{align*}
\int_E h dx &= \int_{E \cap B_a} h dx + \int_{E \setminus B_a} h(|
x|)dx\\
&\ge \int_{E \cap B_a} h dx +\int_{E \setminus B_a} h(|T(x)|)dx\\
&= \int_{E \cap B_a} h dx+\int_{B_a \setminus E} h(|x|)dx\\
&=\int_{B_a} h dx.
\end{align*}
The above implies
$$
\int_E h dx-\int_{B_a} h dx \ge \int_{E \setminus B_a} [h(|x|)-
h(a)]dx
$$
via $|T(x)| \le a$.
Moreover, \cite{MR2672283} yields the inequality for $\mathcal{H}
^{n-1}(\partial^* E)$.
If
$$
h=h_J+h_A+h_c
$$
is the Lebesgue decomposition, $h_J$ is the jump function, $h_A$ is
absolutely
continuous, and $h_c$ is continuous and $h_c'(t)=0$ for a.e. $t$.
Because $h_J'(t)=0$ for a.e. $t$,
$$
h'(t)=h_A'(t)
$$
for a.e. $t$. In particular,
$$
h_A(|x|)-h_A(a)=\int_a^{|x|} h'(t)dt
$$
and this then implies
\begin{align*}
h(|x|)-h(a)&=h_A(|x|)-h_A(a)+h_J(|x|)-h_J(a)+h_c(|x|)-h_c(a)\\
& \ge h_A(|x|)-h_A(a)=\int_a^{|x|} h'(t)dt,
\end{align*}
therefore the statement in the theorem yields (ii).
Observe that if $h$ is convex, $h''(t) \ge 0$ for a.e. $t$. Thus
$h'$ is non-decreasing
which then implies (iii).
Next set $E_*=\text{proj}_{\mathbb{S}^{n-1}}(E \setminus B_a)$,
$E_\sigma=E\cap
L_\sigma$, $L_\sigma$ the line generated in terms of $\sigma \in
\mathbb{S}^{n-1}$.
\noindent \emph{Case 1: $E$ bounded}\\
Assume $E \subset B_{R}$, $r_{-,j}(\sigma), r_{+,j}
(\sigma) \ge a$ so that
\begin{align*}
|E \setminus B_a|&=\int_{E_*} \sum_j \int_{r_{-,j}(\sigma)}^{r_{+,j}
(\sigma)} r^{n-1} drd\sigma.
\end{align*}
Note that
\begin{align*}
r_{+,j}^n-r_{-,j}^n&=n(r_{-,j})^{n-1}(r_{+,j} -r_{-,j})+\cdots+
(r_{+,j} -r_{-,j})^n\\
& \le t_n(r_*) (r_{+,j} -r_{-,j})
\end{align*}
$$t_n(r_*)\ge \max \{n(r_{-,j})^{n-1},\ldots, (r_{+,j}
-r_{-,j})^{n-1}\},$$
$r_* \le R$;
\begin{align*}
\Big[\frac{1}{2} |E \Delta B_a|\Big|^{2}&=|E \setminus B_a|^2\\
&= \Big[\int_{E_*} \sum_j \int_{r_{-,j}(\sigma)}^{r_{+,j}(\sigma)}
r^{n-1} dr d\sigma
\Big]^2\\
&= \Big[\int_{E_*} \frac{1}{n}\sum_j(r_{+,j}^n(\sigma)-r_{-,j}
^n(\sigma))
d\sigma\Big]^2\\
&\le \Big[\frac{t_n(r_*)}{n} \int_{E_*}\sum_j (r_{+,j}(\sigma)
-r_{-,j}
(\sigma))d\sigma\Big]^2\\
&= \Big[\frac{t_n(r_*)}{n}\int_{E_*} {\mathcal{H}^1(E_\sigma)}
d\sigma \Big]^2 \\
&\le \frac{\mathcal{H}^{n-1}(\mathbb{S}^{n-1})t_n^2(r_*)}{n^2}
\int_{E_*}(\mathcal{H}
^1(E_\sigma))^2d\sigma\\
&\le \frac{\mathcal{H}^{n-1}(\mathbb{S}^{n-1})2t_n^2(r_*)}{a^{n-1}
n^2} \int_{E_*}
\int_0^{\mathcal{H}^1(E_\sigma)}r[r+a]^{n-1}drd\sigma \\
&=\frac{\mathcal{H}^{n-1}(\mathbb{S}^{n-1})2t_n^2(r_*)}{a^{n-1}n^2}
\int_{E_*}
\int_a^{a+\mathcal{H}^1(E_\sigma)}[r-a][r]^{n-1}drd\sigma \\
&\le \frac{\mathcal{H}^{n-1}(\mathbb{S}^{n-1})2t_n^2(r_*)}{a^{n-1}
n^2}
\sum_j\int_{E_*} \int_{r_{-,j}(\sigma)}^{r_{+,j}(\sigma)}[r-a]
[r]^{n-1}drd\sigma \\
&= \frac{\mathcal{H}^{n-1}(\mathbb{S}^{n-1})2t_n^2(r_*)}{a^{n-1}n^2}
\int_{E
\setminus B_a} [|x|-a] dx.
\end{align*}
Hence
set
$$
A_*=\left[\frac{\mathcal{H}^{n-1}(\mathbb{S}^{n-1})2t_n^2(r_*)}
{a^{n-1}n^2}\right]^{-1},
$$
it then follows that
$$
\mathcal{G}(E)-\mathcal{G}(B_a) \ge \partial_+h(a)A_*\Big[\frac{1}
{2} |E \Delta B_a|
\Big|^{2}.
$$
Next, define
$$
r_a=2\Big[\frac{1}{\partial_+h(a)A_*}\Big]^{\frac{1}{2}}.
$$
Then
$$
r_a\Big[\mathcal{G}(E)-\mathcal{G}(B_a)\Big]^{\frac{1}{2}} \ge |E
\Delta B_a|.
$$
Now let $x$ satisfy
$$
\frac{|E \Delta (B_a+x)|}{|E|} \le a(m,n)[\mathcal{H}^{n-1}(\partial
E)-\mathcal{H}^{n-1}
(\partial B_a)]^{\frac{1}{2}},
$$
$a(m,n)= \frac{\mu_n}{\sqrt{m}}$ and $\mu_n>0$ depends on the dimension \cite{MR2672283};
supposing
$$
\Big[\mathcal{E}(E)-\mathcal{E}(B_a)\Big] \ge s>0,
$$
where $s$ is a constant, observe
$$
\frac{|E \Delta (B_a+x)|}{|E|} \le 2,
$$
$$
|E \Delta B_a| \le 2m,
$$
yield the lower bound in (iv) with two constants depending on $s$,
$m$.
In particular, one may assume
$$
\Big[\mathcal{E}(E)-\mathcal{E}(B_a)\Big]
$$
is small (observe the constant $s>0$ is arbitrary, therefore one
may let $s$ be
any small constant). Supposing $|x|$ is small,
$$
a_{1,m}|x| \le |(B_a+x) \Delta B_a| \le a_{m}|x|,
$$
$$
\frac{a_{m}}{a_{1,m}} \le w_1
$$
for a universal $w_1>0$; 
\begin{align*}
a_{1,m}|x| &\le |(B_a+x) \Delta B_a| \\
&\le |E \Delta (B_a+x)|+|E\Delta B_a|\\
&\le a(m,n)\big[[ \mathcal{H}^{n-1}(\partial E)-\mathcal{H}^{n-1}
(\partial B_a)]
\big]^{\frac{1}{2}} |E|+r_a\Big[\mathcal{G}(E)-\mathcal{G}(B_a)
\Big]^{\frac{1}{2}}.
\end{align*}
In particular
\begin{align*}
|E\Delta B_a|& \le |E\Delta (B_a+x)| +|(B_a+x)\Delta B_a|\\
&\le a(m,n)\big[ \mathcal{H}^{n-1}(\partial E)-\mathcal{H}^{n-1}
(\partial B_a)
\big]^{\frac{1}{2}} |E|+a_{m}|x|\\
&\le a(m,n)\big[ \mathcal{H}^{n-1}(\partial E)-\mathcal{H}^{n-1}
(\partial B_a)
\big]^{\frac{1}{2}} |E|\\
&+\frac{a_{m}}{a_{1,m}}\Big[ a(m,n)\big[ \mathcal{H}^{n-1}(\partial E)-
\mathcal{H}^{n-1}
(\partial B_a)\big]^{\frac{1}{2}} |E|+r_a\Big[\mathcal{G}(E)-
\mathcal{G}(B_a)
\Big]^{\frac{1}{2}} \Big].
\end{align*}
One can now remove the assumption $E \subset B_R$. \\
\noindent \emph{Case 2: $E$ unbounded}\\
\noindent The initial part is to apply \cite{MR2672283} and obtain
the existence of
some $L \subset E$ that approximates $E$. There is a constant $e>0$
such that if
$$
\mathcal{E}(E)-\mathcal{E}(B_a) \le e
$$
$$
B_a+k \approx L,
$$
$\&$ $|E \setminus L|$ is bounded by $q \sqrt{\mathcal{E}(E)-
\mathcal{E}(B_a)}.$
Therefore if the translation $k$ has a uniform upper bound, one can
let $R$ depend
on $e, m, n$. The bound exists via:
suppose $k \rightarrow \infty$,
\begin{align*}
\infty &=\int_{B_a} \liminf g(y+k)dy \\
&\le \liminf \int_{B_a} [g(y+k)-g(y)]dy+\int_{B_a} g(y)dy\\
&\le w_1 \Big[\mathcal{E}(E)-\mathcal{E}(B_a)\Big]^{1/2}+||g||
_{L^1(B_a)}\\
&\le w_1e^{1/2}+||g||_{L^1(B_a)}
\end{align*}
hence this contradicts $||g||_{L^1(B_a)}<\infty$ since $g$ is nondecreasing
and
convex ($g(0)=0$, therefore without loss of generality, $g(x)>0$ for
an $x \in
\mathbb{R}^n$, which then implies that $g(x) \rightarrow \infty$
when $|x| \rightarrow
\infty$).
In this case, $L \approx E$, $L \subset B_R$, therefore the argument
above implies
$r_* \le R$ with a uniform bound on $R$.
\end{proof}
\begin{rem}
Inequality (iv) has the optimal exponent: suppose to simplify the
explicit calculation
that $n=1$, $B=(-1/2,1/2)$,
assume $E=B+x$,
$|x|$ sufficiently small; then the lower bound is of the order $|x|
$, and the free
energy for $g(a)=a^2$ is of the order
\begin{align*}
\mathcal{G}(B+x)-\mathcal{G}(B) &=\int_{-1/2}^{1/2} (g(x+q)-
g(q))dq=x^2.\\
\end{align*}
\end{rem}
\begin{rem}
The assumption in the classical Almgren problem is convexity of the
potential. Indrei constructed an example of a convex potential for
which one does not have existence of solutions. Thus necessarily one
requires more assumptions. In the theorem above, we obtain that
balls minimize the energy also when $g$ is non-decreasing, radial,
and possibly non-convex. Interestingly, that the minimizers are
balls is new. Also, the non-decreasing assumption is sharp: in the general case, convexity (of minimizers) is not true (e.g. one can construct a potential with a disconnected zero level).
\end{rem}
\begin{rem}
The potential has the byproduct of generating a quantitative
inequality without
translation invariance. Thus the quantitative anisotropic
isoperimetric inequality is
strongly different. One similar theorem is in \cite{MR3023863} where
translations are
crucial when one considers the quantitative isoperimetric inequality
in
convex cones: assuming the cone contains no line, the quantitative
term is without
translations. The characteristic also appears in the exponents of
the mass. In \cite{I20}, a weighted relative isoperimetric
inequality was proved.
\end{rem}
\begin{rem}
The bound in (iii) has the first moment of $E \setminus B_a$: $$
\int_{E
\setminus B_a} [|x|-a] dx=m_1(E \setminus B_a)-a|E \setminus B_a|.$$
\end{rem}
\begin{cor} \label{@4'}
Suppose $f(\nu)=1$ if $\nu \in \mathbb{S}^{n-1}$, $g(x)=h(|x|)$, $g:
\mathbb{R}^n
\rightarrow \mathbb{R}^+$ is non-decreasing and homogeneous of
degree $\alpha$. Let $m>0$ and assume $E_m$ is the minimizer with $|
E_m|=m$, set
$$
m_\alpha=|B_1|\Big[\frac{(n-1)\mathcal{H}^{n-1}(\partial B_1) }
{\alpha(\alpha+n)
\int_{B_1}h(|x|)dx}\Big]^{\frac{n}{1+\alpha}},
$$
it then follows that
$$
m \rightarrow \mathcal{E}(E_m)
$$
is concave on $(0,m_\alpha)$ and convex on $(m_\alpha, \infty)$.
\end{cor}
\begin{proof}
Observe $E_m=B_r$, $r=\Big[ \frac{m}{|B_1|}\Big]^{\frac{1}{n}}$,
which therefore
implies
$$
\mathcal{E}(E_m)=\mathcal{H}^{n-1}(\partial B_1)(\frac{1}{|B_1|
^{\frac{n-1}
{n}}})m^{\frac{n-1}{n}}+(\frac{1}{|B_1|^{\frac{\alpha+n}{n}}})
(\int_{B_1} h(|
x|)dx)m^{\frac{n+\alpha}{n}}.
$$
The critical mass $m_\alpha$ therefore may be calculated via the
second derivative.
\end{proof}
\begin{rem}
The existence of a critical mass was already discussed in
\cite{MR1641031},
nevertheless a specific value was not calculated nor was there some
proof of the
existence of the mass.
\end{rem}
\begin{prop} \label{Ko}
If $m>0$, $g \in L_{loc}^\infty$, and up to sets of measure zero,
$g$ admits unique
minimizers $E_m$, then for $\epsilon>0$ there exists
$w_m(\epsilon)>0$ such that if
$|E|=|E_m|$, $E \subset B_R$, $R=R(m)$, and
$$
\mathcal{A}(E,E_m):=|\mathcal{F}(E)-\mathcal{F}(E_m)+\int_E g(x)dx-
\int_{E_m}
g(x)dx|<w_m(\epsilon),
$$
then
$$
\frac{|E_m \Delta E|}{|E_m|}<\epsilon.
$$
\end{prop}
\begin{proof}
Assume not, then there exists $\epsilon>0$ and for $w>0$, there
exist $E_w' \subset
B_R$ and $E_m$, $|E_w'|=|E_m|=m$,
$$
\mathcal{A}(E_w',E_m)<w,
$$
but
$$
\frac{|E_m\Delta E_w'|}{|E_m|} \ge \epsilon;
$$
in particular, set $w=\frac{1}{k}$, $k \in \mathbb{N}$; observe that
there exist
$E_{\frac{1}{k}}'$, $|E_{\frac{1}{k}}'|=m$,
$$
|\mathcal{E}(E_m)-\mathcal{E}(E_{\frac{1}{k}}')|\le \mathcal{A}
(E_{\frac{1}
{k}}',E_m)<\frac{1}{k},
$$
$$
\frac{|E_m\Delta E_{\frac{1}{k}}'|}{|E_m|} \ge \epsilon;
$$
thus
$$
\mathcal{F}(E_{\frac{1}{k}}') \le \mathcal{E}(E_{\frac{1}{k}}') <
\frac{1}{k}+\mathcal{E}
(E_m),
$$
$$E_{\frac{1}{k}}' \subset B_R,$$
and compactness for sets of finite perimeter implies --up to a
subsequence--
$$
E_{\frac{1}{k}}' \rightarrow E' \hskip .3in in \hskip .08in
L^1(B_R),
$$
and therefore $|E'|=m$ because of the triangle inequality in
$L^1(B_R)$; thus
$$
\mathcal{E}(E') \le \liminf_k \mathcal{E}(E_{\frac{1}{k}}')
=\mathcal{E}(E_m),
$$
which implies
$E'$ is a minimizer contradicting
$$
\frac{|E_m\Delta E'|}{|E_m|} \ge \epsilon
$$
via the uniqueness of minimizers.
\end{proof}
Identifying the modulus $w_m(\epsilon)$ is in general very
difficult. The subsequent
theorem is about sharp error estimates for the unique minimizer
regime.
\begin{thm} \label{@7z'}
Suppose $m>0$, and up to sets of measure zero, $g$ admits unique
minimizers
$E_m \subset B_{r(m)}$,
then\\
\noindent (i) there exists $\alpha_m>0$ s.t. if $g \in W_{loc}
^{1,1}$ is differentiable, $
\epsilon< \alpha_m$,
$$
w_m(\epsilon) \le \lambda(m, \epsilon)=o(\frac{\epsilon m}
{a_{m_*}}),
$$
$$
a_{m_*}=\alpha_*|D \boldsymbol{1}_{E_m} \cdot w_*|(\mathbb{R}^n)
$$
with a universal constant $\alpha_*>0$;\\
\noindent (ii)
there exists $m_a>0$ s.t. if $n=2$, $m<m_a$, $g \in W_{loc}
^{1,\infty}$, then
$$\lambda(m, \epsilon)= \alpha m^{\frac{1}{2}}o_m(\epsilon),$$
$$o_m(\epsilon) =\epsilon \int_{E_m} o_x(1) dx,$$
$$
\int_{E_m} o_x(1) dx \rightarrow 0
$$
as $\epsilon \rightarrow 0^+$;
and assuming $g \in C^1$,
$$\lambda(m, \epsilon)= \alpha m^{\frac{3}{2}}o(\epsilon),$$
$\alpha=\alpha(||D g||_{L^\infty(B_{r(m_a)})})>0;$
\\
\noindent (iii) there exist $m_a, \alpha_m>0$ s.t. if $n \ge 3$,
$m<m_a$, $\epsilon<
\alpha_m$, $g \in C^1$, $f$ is $\lambda_a$-elliptic, then
$$\lambda(m, \epsilon)= \alpha m^{1+\frac{1}{n}}o(\epsilon),$$
$\alpha=\alpha(||D g||_{L^\infty(B_{r(m_a)})}, \lambda_a)>0$;\\
\noindent (iv) there exists $\alpha_m>0$ s.t. if $g \in W_{ loc}
^{2,1}$ is twice
differentiable, $\epsilon< \alpha_m$,
$$
w_m(\epsilon) \le a_m \epsilon^2,
$$
$$a_m=\Big(||D^2 g||_{L^1(E_m)} +m\frac{1}{6} \Big) (\frac{m}
{\alpha_*|D
\boldsymbol{1}_{E_m} \cdot w_*|(\mathbb{R}^n)} )^2$$
with $\alpha_*>0$ universal;\\
(v) there exists $m_a>0$ s.t. if $n=2$, $g \in W_{loc}^{2, \infty}$,
$m<m_a$,
$$
w_m(\epsilon) \le a_m \epsilon^2,
$$
$$
a_m=a_*m^{2},
$$
$a_*=a_*(||D^2 g||_{L^\infty(B_{r(m_a)})})>0$;\\
(vi) there exists $m_a>0$ s.t. if $f$ is $\lambda_a$-elliptic, $g
\in W_{loc}^{2, \infty}$,
$m<m_a$,
$$
w_m(\epsilon) \le a_m \epsilon^2,
$$
$$
a_m=a_*m^{1+\frac{2}{n}},
$$
$a_*=a_*(||D^2 g||_{L^\infty(B_{r(m_a)})}, \lambda_a)>0$.
\end{thm}
\begin{proof}
For (i),
$$y \in (\int_{E_m} \nabla g(x)dx)^{\perp},$$
\begin{align*}
\Big|\int_{E_m+y} g(x)dx-\int_{E_m} g(x)dx\Big|\\
&= \int_{E_m} g(x+y)-g(x)dx\\
&=y\cdot \int_{E_m} \nabla g(x)dx+\int_{E_m} o_x(|y|)dx\\
&=\int_{E_m} o_x(|y|)dx.
\end{align*}
Now there is an $\alpha(x)>0$ s.t. for $|y|<\alpha(x)$,
$$\frac{|o_x(|y|)|}{|y|} \le 1;$$
also, $x \rightarrow \alpha(x)$ is continuous (via the
differentiability of $g$). In
particular, via the compactness of $\overline{E_m}$,
$$
\alpha:=\inf_{E_m} \alpha(x)>0;
$$
therefore, if $|y|<\alpha$,
$$\sup_{x \in E_m} \frac{|o_x(|y|)|}{|y|} \le 1;$$
hence the dominated convergence theorem implies
$$
\limsup_{|y| \rightarrow 0} \int_{E_m} o_x(|y|)/|y| dx =0.
$$
In particular, there exists $\alpha>0$, such that with $|y|<\alpha$
$$
\Big|\int_{E_m+y} g(x)dx-\int_{E_m} g(x)dx\Big|= o(|y|);
$$
observe there exists $t_{E_m}>0$ such that
$$
a_{m_*}|y| \le |(E_m+y) \Delta E_m|
$$
when $|y| \le t_{E_m}$ \cite[Lemma 3.2]{MR3023863},
$$
a_{m_*}=\alpha_*|D \boldsymbol{1}_{E_m} \cdot w_*|(\mathbb{R}^n),
$$
$w_* \in \mathbb{R}^n$;
therefore there then exists $\alpha_{m}=t_{E_m} a_{m_*}/m>0$ so that
if $\epsilon <
\alpha_{m}$,
$|y|$ may be taken so that
$$
\frac{a_{m_*}}{m}|y|=\epsilon;
$$
hence assuming $\epsilon < \alpha_{m}$,
$$
w_m(\epsilon) \le \Big|\int_{E_m+y} g(x)dx-\int_{E_m} g(x)dx\Big|
=o(|
y|)=o(\frac{\epsilon m}{a_{m_*}})
$$
thus proving (i). Observe that in order to prove (ii),
it is then sufficient to show
\begin{equation} \label{es}
a_{m_*}= C \mathcal{H}^{n-1}(\partial E_m),
\end{equation}
for a constant independent of mass. In order to see why, supposing
this is valid
\begin{align*}
a_{m_*}&=C\mathcal{H}^{n-1}(\partial E_m)\\
& \approx \mathcal{H}^{n-1}(\partial \mu_m K) \\
&= \mu_m^{n-1}\mathcal{H}^{n-1}(\partial K)= (\frac{m}{|
K|})^{\frac{n-1}{n}}
\mathcal{H}^{n-1}(\partial K)
\end{align*}
via $|\mu_m K|=m$ (where $K$ is the Wulff shape).
Note that one can control the $o_x(|y|)$ remainder as above:
$$ \sup_{x \in E_m} \frac{|o_x(|y|)|}{|y|} \le 2A$$(since $g \in W_{loc}
^{1,\infty}$, $A=||\nabla g||
_{L^\infty(B_r)}$).
Hence the dominated convergence theorem implies
$$
| y| \int_{E_m} o_x(|y|)/|y| dx = |y|o_m(1) =\frac{\epsilon m}
{a_{m_*}}o_m(1)
$$
\begin{align*}
(\frac{\epsilon m}{a_{m_*}})o_m(1)&=(\frac{\epsilon m}{\bar{C}
m^{\frac{n-1}
{n}}})o_m(1)\\
&=\frac{o_m(\epsilon)}{\bar{C}} m^{1/n};
\end{align*}
supposing $g \in C^1$,
$$
\sup_{x\in E_m} \frac{o_x(|y|)}{|y|} \rightarrow 0,
$$
as $|y| \rightarrow 0$; thus
$$
o_m(1)=m o(1),
$$
$o(1) \rightarrow 0$ as $\epsilon \rightarrow 0$.
In particular
$$m (\frac{\epsilon m}{a_{m_*}})o(1)=m (\frac{\epsilon m}{\bar{C}
m^{\frac{n-1}
{n}}})o(1)=\frac{o(\epsilon)}{\bar{C}} m^{1+1/n}.
$$
Assuming $n=2$, $m$ is small, the minimizer $E_m$ is convex
\cite{MR2807136}.
In particular, it suffices to prove \eqref{es} when $E_m$ is convex in general.
Let $A=E_m$.
For $t >0$ and $w \in \mathbb{S}^{n-1}$, define $$f_w(t):=|(A+tw)
\Delta A|.$$ Note
$$f_w(t) = \int_{\mathbb{R}^n} |\mathbf{1}_{(A+ tw)}(x) - \mathbf{1}
_{A}(x)| dx =
\int_{\mathbb{R}^n} |\mathbf{1}_{A}(x+tw) - \mathbf{1}_{A}(x)| dx.$$
Furthermore, for $t>0$, set $$g_w(t):=\int_{\partial(A-tw)\cap A}
\displaystyle \langle
\nu_{A}(x+tw),w \rangle d\mathcal{H}^{n-1}(x),$$ where $\nu_{A}$ is
the $\it{outer}$
unit normal to the boundary of $A$.
\noindent To begin: for $s>0$, if $0<t < s$, then
\begin{equation} \label{n47}
\partial(A-sw)\cap A +(s-t)w \subset \partial(A-tw)\cap A.
\end{equation}
Indeed, set $x=y+(s-t)w$ with $y \in \partial(A-sw)\cap A$. Thus,
$y=a-sw,$ with
$a\in \partial A$ and $x=a-tw$. Since $y \in A$, convexity implies
that for all $\lambda
\in (0,1)$, $$\lambda y+(1-\lambda)a \in A.$$ But
$\lambda y+(1-\lambda)a= \lambda (a-sw)+(1-\lambda)a=a-\lambda sw$.
Letting $
\lambda=\frac{t}{s}$, $a-tw=x \in A$,
and this completes the proof of the claim. Note that by convexity of
$A$,
$$\displaystyle \langle \nu_{A}(x+tw),w \rangle \geq 0$$
for all
$x \in (\partial A-tw)\cap A$
(if $y \in \partial A$ and $\nu_{A}(y)\cdot w<0$, then $y-tw \in \{z :\langle z-y,\nu_A(y)
\rangle >0\}$, and the latter is disjoint from $A$).
Therefore, by the change of variable $x=y+(s-t)w$ and (\ref{n47}),
\begin{align*}
g_w(s)&= \int_{\partial(A-sw)\cap A} \displaystyle \langle \nu_{A}
(y+sw),w \rangle
d\mathcal{H}^{n-1}(y)\\
&= \int_{\partial(A-sw)\cap A +(s-t)w} \displaystyle \langle \nu_{A}
(x+tw),w \rangle
d\mathcal{H}^{n-1}(x)\\
&\leq g_w(t),
\end{align*}
provided $0<t<s$. Next, for a particular value of $s=s(n, A)>0$,
\begin{equation} \label{rrr03}
\displaystyle \inf_{w \in \mathbb{S}^{n-1}} g_w(s) >0.
\end{equation}
\noindent Indeed, let $w \in \mathbb{S}^{n-1}$, and for $y \in
\partial A$ find $r_y>0$
so that $B_{r_y}(y) \cap \partial A$ is the graph of a concave
function $u_y:
\mathbb{R}^{n-1} \rightarrow \mathbb{R}$. Upon a possible
reorientation, we have
$$A \cap B_{r_y}(y)=\{x \in B_{r_y}(y) : x_n <
u_y(x_1,\ldots,x_{n-1})\}.$$ $\tilde{B}
_{\beta}(\alpha)$ denotes the open ball in $\mathbb{R}^{n-1}$
centered at $\alpha$
with radius $\beta$ in the following argument.
Since $$\partial A \subset \displaystyle \bigcup_{y \in \partial A}
B_{r_y/2}(y),$$ by
compactness of $\partial A$ there exists $N \in \mathbb{N}$, $\{y_k\}_{k=1}^{N}
\subset \partial A $ such that $\partial A \subset \cup_{k=1}^N B_{r_k/2}(y_k)$.
Denote the corresponding concave functions by $\{u_k\}_{k=1}^{N}$,
with the
consensus that $\partial A \cap B_{r_k}(y_k)$ is the graph of $u_k$.
Next, let $y^* \in
\partial A$ be such that the normal to a supporting hyperplane at
$y^*$ is $w$.
Then $y^* \in \partial A \cap B_{r_j/2}(y_j)$ for a $j \in
\{1,2,\ldots,N\}$. Moreover,
$u_j : \tilde{B}_{r_j}(\hat{x}_j) \rightarrow \mathbb{R}$, where $
(\hat{x}_j, u_j(\hat{x}))
= y_j$,
\begin{equation}
A \cap B_{r_j}(y_j)=\{x \in B_{r_j}(y) : x_n < u(x_1,\ldots,x_{n-1})
\}.
\end{equation}
Next, let $\hat{x}^* \in \tilde{B}_{r_j/2}(\hat{x}_j)$ be such that
$(\hat{x}^*, u_j(\hat{x}
^*))=y^*$, $r=r(n, A):= \displaystyle \min_k r_k/4$, \& $s=s(n,
A):=r/4$. Since $\hat{x}
^* \in \tilde{B}_{r_j/2}(\hat{x}_j)$, $\tilde{B}_{r}(\hat{x}^*)
\subset \tilde{B}_{3r_j/4}
(\hat{x}_j)$. Denote the superdifferential of a concave function $
\phi: \mathbb{R}
^{n-1} \rightarrow \mathbb{R}$ at a point $\hat{x}$ in its domain by
$$\partial^{+}
\phi(\hat{x}) := \{\hat{y} \in \mathbb{R}^{n-1} : \forall \hat{z}
\in \mathbb{R}^{n-1},
\phi(\hat{z}) \leq \phi(\hat{x})+\langle \hat{y}, \hat{z}-\hat{x}
\rangle \}.$$ Because $w$
is the normal to a supporting hyperplane of $A$ at $(\hat{x}^*,
u_j(\hat{x}^*))$, there
exists $\nabla^{+} u_j(\hat{x}^*) \in \partial^{+} u_j(\hat{x}^*)$
such that $w=\frac{(-
\nabla^{+} u_j(\hat{x}^*),1)}{\sqrt{1+|\nabla^{+} u_j(\hat{x}^*)|
^2}}=:(\hat{w}, w^\perp)
$. Up to an infinitesimal rotation, without loss $|\nabla^{+}
u_j(\hat{x}^*)| >0$. Let
$e_1 = \frac{\nabla^{+} u_j(\hat{x}^*)}{|\nabla^{+} u_j(\hat{x}^*)|}
$ \&
$\{e_2,\ldots,e_{n-1}\}$ be a corresponding orthonormal basis for $
\mathbb{R}^{n-1}$.
Define $$N_{\epsilon}:= \bigg \{ \hat{z} \in \tilde{B}_{r}(\hat{x}
^*) : \hat{z}-\hat{x}
^*=\gamma_1 e_1 + \sum_{i=2}^{n-1} \gamma_i e_i, \hskip .1in \big |
\sum_{i=2}
^{n-1} \gamma_i e_i \big| \leq \epsilon, \hskip .1in -\frac{r}{4}
\geq \gamma_1\bigg\},
$$ and $$\tilde{N}_{\epsilon}:= \bigg \{ \hat{z} \in \tilde{B}_{r}
(\hat{x}^*) : \hat{z}-
\hat{x}^*=\gamma_1 e_1 + \sum_{i=2}^{n-1} \gamma_i e_i, \hskip .1in
\big |
\sum_{i=2}^{n-1} \gamma_i e_i \big| \leq \epsilon, \hskip .1in -
\frac{r}{2} \geq
\gamma_1 \bigg \}.$$
These are rectangular sets inside a region.
Let $\hat{z} \in \tilde{N}_\epsilon$, and note
\begin{align}
u_j(\hat{z}) &\leq u_j(\hat{z}-s\hat{w})+\langle
\nabla^+u_j(\hat{z}-s\hat{w}), s\hat{w}
\rangle \nonumber \\
&=u_j(\hat{z}-s\hat{w})-\frac{s}{\sqrt{1+|\nabla^+u_j(\hat{x}^*)|
^2}} \langle
\nabla^+u_j(\hat{z}-s\hat{w}), \nabla^+u_j(\hat{x}^*) \rangle.
\label{fin8}
\end{align}
Furthermore, if $\hat{z} \in \tilde{B}_{r}(\hat{x}^*)$, there exists
$\nabla^{+}
u_j(\hat{z}) \in \partial^{+} u_j(\hat{z})$ such that $\nu_{A}
((\hat{z},
u_j(\hat{z})))=\frac{(-\nabla^{+} u_j(\hat{z}), 1)}{\sqrt{1+|
\nabla^{+} u_j(\hat{z})|^2}}.$
Thus,
\begin{align}
\langle \nu_{A}((\hat{z}, u_j(\hat{z}))), w \rangle &= \bigg \langle
\frac{(-\nabla^{+}
u_j(\hat{z}), 1)}{\sqrt{1+|\nabla^{+} u_j(\hat{z})|^2}}, \frac{(-
\nabla^{+} u_j(\hat{x}^*),1)}
{\sqrt{1+|\nabla^{+} u_j(\hat{x}^*)|^2}} \bigg \rangle \nonumber\\
&= \frac{1+ \nabla^{+} u_j(\hat{z}) \cdot \nabla^{+} u_j(\hat{x}^*)}
{\sqrt{1+|\nabla^{+}
u_j(\hat{z})|^2} \sqrt{1+|\nabla^{+} u_j(\hat{x}^*)|^2}}
\label{aaa2n}
\end{align}
If $\tilde{C}=\displaystyle \max_k \{Lip_{\frac{3}{4}}(u_k)\}>0$,
where $Lip_{\frac{3}
{4}}(u_k)$ denotes the Lipschitz constant of $u_k$ on $\tilde{B}
_{\frac{3}{4}r_k}
(\hat{x}_k)$ (recall $\hat{x}_k:=u_k^{-1}(y_k)$), then since $
\displaystyle
\sup_{\hat{z} \in \tilde{B}_{r}(\hat{x}^*)} |\nabla^{+}
u_j(\hat{z})| \leq \tilde{C}$,
\begin{equation} \label{s2b2}
\frac{1}{\sqrt{1+|\nabla^{+} u_j(\hat{z})|^2} \sqrt{1+|\nabla^{+}
u_j(\hat{x}^*)|^2}} \geq
\frac{1}{1+\tilde{C}^2}=:C_0.
\end{equation}
Next, the monotonicity formula for the superdifferential of the
concave function $u_j$
yields
$$\langle \nabla^{+} u_j(\hat{z}) - \nabla^{+} u_j(\hat{x}
^*), \hat{z}-\hat{x}^*
\rangle \leq 0,$$
and since
$$\langle \nabla^{+} u_j(\hat{x}^*),
\hat{z}-\hat{x}^* \rangle
\leq 0$$
for $\hat{z} \in N_\epsilon$,
\begin{align*}
0& \geq \langle \nabla^{+} u_j(\hat{z}), \hat{z}-\hat{x}^* \rangle =
\langle \nabla^{+}
u_j(\hat{z}), \gamma_1 e_1 + \sum_{i=2}^{n-1} \gamma_i e_i \rangle\\
&\geq \frac{\gamma_1}{|\nabla^{+} u_j(\hat{x}^*)|} \langle
\nabla^{+} u_j(\hat{z}),
\nabla^{+} u_j(\hat{x}^*) \rangle - C_0\epsilon.
\end{align*}
As $0<\frac{r}{4} \leq -\gamma_1$,
\begin{equation} \label{a3r}
\langle \nabla^{+} u_j(\hat{z}), \nabla^{+} u_j(\hat{x}^*) \rangle
\geq -
\frac{4C_0\epsilon|\nabla^{+} u_j(\hat{x}^*)|}{r} \geq -\frac{4C_0
\tilde{C} \epsilon}{r}.
\end{equation}
Note that if $\hat{z} \in \tilde{N}_\epsilon$, $\hat{z}-s\hat{w} \in
N_\epsilon$ so by
(\ref{a3r}) $$\langle \nabla^{+} u_j(\hat{z}-s\hat{w}), \nabla^{+}
u_j(\hat{x}^*) \rangle
\geq -\frac{4C_0 \tilde{C} \epsilon}{r}.$$ Combining this with
(\ref{fin8}),
\begin{equation} \label{fin9j}
u_j(\hat{z}-s\hat{w})\geq u_j(\hat{z}) -s\frac{4C_0 \tilde{C}
\epsilon}{r\sqrt{1+|
\nabla^+u_j(\hat{x}^*)|^2}}\geq u_j(\hat{z}) -s\frac{4C_0 \tilde{C}
\epsilon}{r}.
\end{equation}
Now suppose
$$\epsilon \in \big(0, \min\{ \frac{r}{4C_0\tilde{C}
\sqrt{1+C^2}}, \frac{r}
{8C_0\tilde{C}} \}\big).$$
Note $\epsilon=\epsilon(n,A)$ and with
this choice of $
\epsilon$, (\ref{aaa2n}), (\ref{s2b2}), and (\ref{a3r}) imply that for
$\hat{z} \in \tilde{N}
_\epsilon$,
\begin{equation} \label{fin10b}
\langle \nu_{A}((\hat{z}, u_j(\hat{z}))), w \rangle \geq \frac{1}{2}
C_0,
\end{equation}
whereas (\ref{fin9j}) implies
\begin{equation} \label{fin11r}
u_j(\hat{z}-s\hat{w}) > u_j(\hat{z})-\frac{s}{\sqrt{1+C^2}}\geq
u_j(\hat{z})-sw^\perp.
\end{equation}
From (\ref{fin11r}), $(\hat{z}-s\hat{w}, u_j(\hat{z})-sw^\perp) \in
\partial(A-sw) \cap A$.
Moreover, let $U_s: \mathbb{R}^{n-1} \rightarrow \mathbb{R}^n$ be
defined via
$U_s(\hat{d})=(\hat{d}, u_j(\hat{d}+s\hat{w})-sw^\perp)$. Then
\begin{equation}
\label{hop2r}
U_s(\tilde{N}_\epsilon-s\hat{w}) \subset \partial(A-sw) \cap A.
\end{equation}
Therefore, if $y \in U_s(\tilde{N}_\epsilon - s\hat{w})$ so that
$y=(\hat{z}-s\hat{w},
u_j(\hat{z})-sw^\perp)$ for $\hat{z} \in \tilde{N}_\epsilon$, then $
\nu_{(A-sw)}(y) =
\nu_A((\hat{z}, u_j(\hat{z})))$; hence, (\ref{fin10b}) implies $$
\langle \nu_{(A-sw)}(y), w
\rangle = \langle \nu_{A}((\hat{z}, u_j(\hat{z})), w \rangle \geq
\frac{1}{2}C_0.$$
The aforementioned fact and (\ref{hop2r}) yields
\begin{align}
g_w(s) &= \int_{\partial(A-sw)\cap A} \displaystyle \langle \nu_{Asw}(
y),w \rangle
d\mathcal{H}^{n-1}(y) \nonumber\\
&\geq \int_{U_s(\tilde{N}_{\epsilon}-s\hat{w})} \displaystyle
\frac{1}{2}C_0
d\mathcal{H}^{n-1}(y)=\frac{1}{2}C_0 \mathcal{H}^{n-1}(U_s(\tilde{N}
_{\epsilon}-
s\hat{w})) \nonumber \\
&=\frac{1}{2}C_0 C_1 \mathcal{H}^{n-1}(\partial A)\nonumber\\
&= \frac{1}{2}C_0 C_1 \mathcal{H}^{n-1}(\partial E_m),\nonumber\\
\end{align}
$C_0 C_1$ depending only on the shape of $E_m$ (not the mass) in the sense of
the Lipschitz constant of the Wulff shape, and
this proves
(\ref{rrr03}).
Next, $f'_w(t)=2g_w(t)$ almost everywhere. To see this, set
\begin{equation*}
h(y):=|A \setminus (A-y)|=\int_{\mathbb{R}^n} \mathbf{1}_{A
\setminus (A-y)}(x) dx,
\end{equation*}
and $\psi \in C_c^{\infty}(\mathbb{R}^n; \mathbb{R}^n).$ Now, $h(y)$ is
Lipschitz;
therefore, using integration by parts, Fubini, and that $\psi$ is
divergence free
\begin{align}
\int_{\mathbb{R}^n} \psi \cdot \nabla h(y) dy &=-\int_{\mathbb{R}^n}
\operatorname{div} \psi(y) h(y)dy \nonumber \\
&=-\int_{\mathbb{R}^n} \biggl(\int_{\mathbb{R}^n} \operatorname{div}
\psi(y)
\mathbf{1}_{A \setminus (A-y)}(x) dy \biggr) dx \nonumber\\
&=-\int_{A} \biggl(\int_{\mathbb{R}^n} \operatorname{div} \psi(y)
\mathbf{1}_{A
\setminus (A-y)}(x) dy \biggr) dx \nonumber\\
&=-\int_{A} \biggl(\int_{\mathbb{R}^n} \operatorname{div} \psi(y)
(1-\mathbf{1}_{(A-y)}
(x)) dy \biggr) dx \nonumber\\
&=\int_{A} \biggl(\int_{\mathbb{R}^n} \operatorname{div} \psi(y)
\mathbf{1}_{(A-y)}(x)
dy \biggr) dx. \label{n16}
\end{align}
Next, note $\mathbf{1}_{(A-y)}(x)=\mathbf{1}_{(A-x)}(y)$ and by
applying the
divergence theorem, it follows that
\begin{align}
\int_{\mathbb{R}^n} \operatorname{div} \psi(y) \mathbf{1}_{(A-x)}
(y)dy&=\int_{\mathbb{R}^n} \langle \psi(y), \nu_{(A-x)}(y) \rangle
d\mathcal{H}^{n-1}
\lfloor \partial(A-x)(y) \nonumber\\
&=\int_{\mathbb{R}^n} \langle \psi(z-x), \nu_{A}(z) \rangle
d\mathcal{H}^{n-1} \lfloor
\partial A(z) \label{n22}.
\end{align}
Hence, (\ref{n16}), (\ref{n22}), and Fubini imply
\begin{align*}
\int_{\mathbb{R}^n} \psi \cdot \nabla h(y) dy &=\int_{A}
\biggl(\int_{\mathbb{R}^n}
\langle \psi(z-x), \nu_{A}(z) \rangle d\mathcal{H}^{n-1} \lfloor
\partial A(z)\biggr)dx\\
&=\int_{\mathbb{R}^n} \biggl(\int_{\mathbb{R}^n} \psi(z-x)
\mathbf{1}_A(x) dx \biggr)
\cdot \nu_{A}(z) d\mathcal{H}^{n-1} \lfloor \partial A(z)\\
&=\int_{\mathbb{R}^n} \biggl(\int_{\mathbb{R}^n} \psi(y) \mathbf{1}
_A(z-y) dy \biggr)
\cdot \nu_{A}(z) d\mathcal{H}^{n-1} \lfloor \partial A(z)\\
&=\int_{\mathbb{R}^n} \psi(y) \cdot \biggl(\int_{\mathbb{R}^n}
\mathbf{1}_A(z-y)
\nu_{A}(z) d\mathcal{H}^{n-1} \lfloor \partial A(z)\biggr)dy\\
&=\int_{\mathbb{R}^n} \psi(y) \cdot \biggl(\int_{\partial (A-y)}
\mathbf{1}_A(x) \nu_{A}
(x+y) d\mathcal{H}^{n-1}(x) \biggr)dy.
\end{align*}
But, $\psi \in C_c^{\infty}(\mathbb{R}^n;\mathbb{R}^n)$ is
arbitrary, this hence
implies that for a.e. $y \in \mathbb{R}^n$, $$\nabla
h(y)=\int_{\partial (A-y) \cap A}
\nu_{(A-y)}(x) d\mathcal{H}^{n-1}(x).$$
Note $f_w(t)=2h(tw)$, and since $f_w(t)$ is Lipschitz, for a.e. $t
\in \mathbb{R}$,
$$f'_w(t)=2 \nabla h(tw) \cdot w= 2 g_w(t).$$
Now one may finish the proof: let $y \in \mathbb{R}^n$, and write
$y=tw$ for some
$w \in \mathbb{S}^n$. If $t \in [0,s]$,
\begin{align*}
f_w(t)&=\int_0^t f'_w(\xi)d\xi= \int_0^t 2g_w(\xi)d\xi \geq
\displaystyle 2\inf_{w\in
\mathbb{R}^n}g_w(s) t\\
&= \displaystyle 2\inf_{w\in \mathbb{R}^n} g_w(s) |y|.\\
\end{align*}
Let $C(n,A):= \displaystyle 2 \inf_{w\in \mathbb{R}^n} g_w(s)$, and
note that
$$C(n,A)\ge C_0 C_1 \mathcal{H}^{n-1}(\partial E_m)=:a_{m_*}>0$$
thanks to
(\ref{rrr03}). Thus, for $|y| \leq s=s(n,A)$,
$$a_{m_*}|y| \leq |(A+y) \Delta A|.$$
In particular, (ii) is true.
Assuming $n \ge 3$, if $g \in C^1$ and $f$ is elliptic, $m$ is
small, then $E_m$ is
convex \cite{MR2807136}. Hence (iii) is true.
For the proof of (iv), assume
$$y_0 \in (\int_{E_m} \nabla g(x)dx)^{\perp},$$
\begin{align*}
\Big|\int_{E_m+y_0} g(x)dx-\int_{E_m} g(x)dx\Big|\\
&= \int_{E_m} g(x+y_0)-g(x)dx\\
&=y_0\cdot \int_{E_m} \nabla g(x)dx+\frac{1}{2} \int_{E_m}
\Big( \langle D^2 g(x)y_0,
y_0\rangle +2o_x(|y_0|^2)\Big)dx\\
&=\frac{1}{2} \int_{E_m}( \langle D^2 g(x)y_0, y_0\rangle +2o_x(|
y_0|^2)) dx;\\
\end{align*}
now there is an $\alpha(x)>0$ s.t. for $|y_0|<\alpha(x)$,
$$\frac{|o_x(|y_0|^2)|}{|y|^2} \le 1;$$
also, $x \rightarrow \alpha(x)$ is continuous. In particular, via
the compactness of $
\overline{E_m}$,
$$
\alpha:=\inf_{E_m} \alpha(x)>0;
$$
therefore, if $|y_0|<\alpha$,
$$\sup_{x \in E_m} \frac{|o_x(|y_0|^2)|}{|y_0|^2} \le 1;$$
hence if $|y_0|<\alpha$,
$$
\langle D^2 g(x)\frac{y_0}{|y_0|}, \frac{y_0}{|y_0|}\rangle +|o_x(|
y_0|^2)|/|y_0|^2 \le |
D^2 g(x)| +1
$$
therefore the dominated convergence theorem implies
$$
\limsup_{|y_0| \rightarrow 0} \int_{E_m} (\langle D^2 g(x)\frac{y_0}
{|y_0|}, \frac{y_0}{|
y_0|}\rangle +o_x(|y_0|^2)/|y_0|^2) dx \le \int_{E_m}|D^2g(x)|dx.
$$
In particular, there exists $\alpha>0$, such that with $|y_0|
<\alpha$
\begin{align*}
\Big|\int_{E_m+y_0} g(x)dx-\int_{E_m} g(x)dx\Big|\\
&\le |y_0|^2 \int_{E_m}(|D^2 g(x)| +\frac{1}{6}) dx;\\
&=\mu_{W^{2,1}}|y_0|^2,
\end{align*}
$$
\mu_{W^{2,1}}=\int_{E_m}(|D^2 g(x)| +\frac{1}{6}) dx.
$$
Observe there exists $t_{E_m}>0$ such that
$$
a_{m_*}|y_0| \le |(E_m+y_0) \Delta E_m|
$$
when $|y_0| \le t_{E_m}$, therefore there then exists $\alpha_{m}
=t_{E_m} a_{m_*}/
m>0$ so that if $\epsilon < \alpha_{m}$,
$|y_0|$ may be taken so that
$$
\frac{a_{m_*}}{m}|y_0|=\epsilon;
$$
hence assuming $\epsilon < \alpha_{m}$, (iv) is true via:
\begin{align*}
w_m(\epsilon)& \le \Big|\int_{E_m+y_0} g(x)dx-\int_{E_m} g(x)dx\Big|
\\
& \le \mu |y_0|^2=\Big(||D^2 g||_{L^1(E_m)} +m\frac{1}{6} \Big)
(\frac{\epsilon m}
{a_{m_*}})^2.
\end{align*}
Now assume $g \in W_{loc}^{2,\infty}$, $|y_0|\le q$, \& suppose
$w_{x,y_0}$ is s.t.
(thanks to the mean-value theorem)
\begin{align*}
\frac{1}{2} \int_{E_m}( \langle D^2 g(x)y_0, y_0\rangle +2o_x(|y_0|
^2)) dx&=\frac{1}
{2}|y_0|^2\int_{E_m} \langle D^2 g(w_{x,y_0})\frac{y_0}{|y_0|},
\frac{y_0}{|y_0|}
\rangle dx\\
&\le \frac{1}{2}||D^2g||_{L^{\infty}(a_qE_m)}m |y_0|^2.
\end{align*}
In particular,
$$
w_m(\epsilon) \le \Big|\int_{E_m+y_0} g(x)dx-\int_{E_m} g(x)dx\Big|
\le \mu |y_0|
^2\le\frac{1}{2}||D^2g||_{L^{\infty}(a_qE_m)}m (\frac{\epsilon m}
{a_{m_*}})^2.
$$
Observe that if $m$ is small with $f$ elliptic, $E_m$ is convex and
arbitrarily close
(in terms of the mass) to a re-scaling of the Wulff shape.
Therefore (vi) is proved thanks to:
$$m (\frac{\epsilon m}{a_{m_*}})^2= m (\frac{\epsilon m}{\bar{C}
m^{\frac{n-1}
{n}}})^2=\frac{\epsilon^2}{\bar{C}} m^{1+\frac{2}{n}}.
$$
Assuming $n=2$, the convexity holds without an ellipticity
assumption on $f$ in
particular showing (v).
\end{proof}
\begin{rem} \label{zo}
The modulus in Theorem \ref{@'} is calculated in an explicit way. In
particular,
supposing $m$ is small, $f(w)=1$ if $w \in \mathbb{S}^{n-1}$, $g(|
x|)=|x|^2$,
$w_m(\epsilon)=a_1\epsilon^2 m^3$. Thus the $\epsilon$ dependence is
optimal.
Observe also that the $m$ dependence in (v) supposing $g$ is the
quadratic is
almost sharp (note also that the coefficient is a function of the
mass where one sees a nontrivial dependence, therefore this term is
likely to generate the sharp function of $m$). Assuming $g(x)=|x|$,
$f(w)=1$ if $w \in \mathbb{S}^{n-1}$,
$w_m(\epsilon)=\overline{a_1}\epsilon^2 m^{3-\frac{1}{n}}$.\\
\end{rem}
\begin{rem}
Supposing $n=2$ and that $g$ is strictly convex, then minimizers in
the collection of
convex sets are up to sets of measure zero unique (one proof is with
displacement
interpolation \cite{MR1641031}). Also, assuming
when $E$ is bounded, convex, $0 \notin E$, $|E \cap \{g > 0\}|>0$,
$$
\int_E \nabla g(x) dx \neq 0,
$$
any $E_m$ is convex and up to sets of measure zero unique
\cite{Cryst}.
\end{rem}
\subsection{Curvature estimates}
In this section we discuss some curvature estimates for smooth
minimizers in $
\mathbb R^3.$ The technique is new and we obtain convexity as an
application of the quantitative stability for general perturbations
of the radial potential in the large mass regime.
Let $K$ be the Gauss curvature, then one has
\begin{equation}
2K=2\kappa_1\kappa_2=(\kappa_1+\kappa_2)^2-\kappa_1^2-
\kappa_2^2=H^2-|A|
^2,
\end{equation}
where $H$ is the mean curvature and $A$ is the second fundamental
form, $\kappa_1,
\kappa_2$ are the principal curvatures.
Assuming $f$ a surface tension, the anisotropic mean curvature is 
$$
H_f=\text{trace} \Big(D^2 f A \Big).
$$
The formula for the first variation implies
$$
H_f=\mu-g,
$$
where 
$$
\mu= \frac{(n-1)\mathcal{F}(E_m) + \int_{\partial^* E_m} g \langle x, \nu_{E_m} \rangle d \mathcal{H}^{n-1}}{n|E_m|}.
$$
\begin{thm}
Suppose that $\Sigma\subset\mathbb R^3$ is a smooth minimizer such
that $D^2g$ is bounded
in $\mathbb R^3$, $f(w)=1$ if $w \in \mathbb{S}^{n-1}$, and
\begin{itemize}
\item[(i)] $\mu-g>0$ and small,
\item[(ii)] $\frac{|A|^2}{(\mu-g)^2}\sim 1+\epsilon$, with $\epsilon>0$,
\end{itemize}
then there is a nondecreasing function $\sigma(t)\gg t,$ for small
and positive $t$, such that
\[
0<\sigma(\mu-g)\le |\nabla_\Sigma g|^2
\]
implies that $\Sigma$ is convex.
\end{thm}
\begin{rem}
The condition $(ii)$ above implies
\[
\kappa_1^2+\kappa_2^2\sim(1+\epsilon)(\kappa_1^2+\kappa_2^2+2K),
\]
in other words
\[
\frac{-2K}{\kappa_1^2+\kappa_2^2}\sim \frac\epsilon{1+\epsilon}.
\]
Note that we always have the inequality $\frac{|2K|}{\kappa_1^2+
\kappa_2^2}\le 1$.
\end{rem}
\begin{proof}
Introduce the function
\[
v=H^\alpha(H^2-|A|^2),
\]
where $\alpha$ is a negative number to be chosen below.
Suppose $\min_{x\in \Sigma} v(x)<0$, then there is a point
$x_0\in\Sigma$ such that
this minimum is attained.
Let us choose a local coordinate system at $x_0$ such that $A$ is
diagonal, i.e.
$A=\textrm{diag}[\kappa_1,\kappa_2]$, where $\kappa_1, \kappa_2$ are
the
principal curvatures of $\Sigma$ at $x_0$.
Writing the Laplace-Beltrami operator in local coordinates $
\Delta_{\Sigma} v=g^{ij} (\partial_i \partial_j v-\Gamma_{ij}^k
\partial_k v)$, $\Gamma_{ij}^k$ being the Christoffel symbols, we
see that at $x_0$, where the gradient of $v$ vanishes and the metric
tensor $g_{ij}$ is identity, $\Delta_\Sigma v\ge 0$,
because by assumption $x_0$ is a local minimum
point for $v$. Therefore, using the Simons identity we get
\begin{eqnarray*}
\Delta_\Sigma v
&=&
\Delta_\Sigma (H^{2+\alpha}-H^{\alpha}|A|^2)\\
&=&
(2+\alpha)(1+\alpha)H^\alpha|\nabla_\Sigma H|^2+(2+
\alpha)H^{\alpha+1}\Delta_\Sigma H\\
&&-H^{\alpha}\left[2 |\nabla_\Sigma A|^2-2|A|^4+2\langle A, D^2
H\rangle+2H\textrm{Trace}
(A^3)\right]\\
&&
-2(\alpha H^{\alpha-1}\nabla_\Sigma H\cdot \langle 2A, \nabla_\Sigma
A\rangle)\\
&&
-|A|^2(\alpha(\alpha-1)H^{\alpha-2}|\nabla_\Sigma H|^2+\alpha
H^{\alpha-1}\Delta H)
.
\end{eqnarray*}
Next, notice that $\nabla_\Sigma v=0$ at $x_0$, hence
\[
(\alpha+2)H^{\alpha+1}\nabla_\Sigma H-\alpha
H^{\alpha-1}\nabla_\Sigma H|A|^2
=
2H^\alpha \langle A, \nabla_\Sigma A\rangle,
\]
or equivalently
\[
2H^\alpha\langle A, \nabla_\Sigma A\rangle
=
\nabla_\Sigma H\left(
2H^{\alpha+1}+\alpha H^{\alpha-1}(H^2-|A|^2)\right)
=
\nabla_\Sigma H\left(
2H^{\alpha+1}+\alpha \frac v{H}\right)
\]
which yields
\begin{equation}
2\alpha H^{\alpha-1}\nabla_\Sigma H\cdot \langle A, \nabla_\Sigma
A\rangle=\frac\alpha H|\nabla_\Sigma H|^2\left(
2H^{\alpha+1}+\alpha \frac v{H}\right).
\end{equation}
Furthermore,
\begin{eqnarray*}
|\nabla_\Sigma A|^2
&\ge&
\frac{|\langle A, \nabla_\Sigma A\rangle|^2}{|A|^2}\\
&=&
\frac{1}{4H^{2\alpha}|A|^2} |\nabla_\Sigma H|^2\left(
2H^{\alpha+1}+\alpha \frac v{H}\right)^2.
\end{eqnarray*}
Finally, observe that
\[
\textrm{Trace}(A^3)=\kappa_1^3+\kappa_2^3=(\kappa_1+\kappa_2)
(\kappa_1^2-
\kappa_1\kappa_2+\kappa_2^2)=
H(|A|^2-K).
\]
After rearranging the terms, we get that $v$ satisfies the following
equation
\begin{eqnarray*}
0
&\le&
\Delta_\Sigma (H^{2+\alpha}-H^{\alpha}|A|^2)\\
&\le &
(2+\alpha)(1+\alpha)H^\alpha|\nabla_\Sigma H|^2+(2+
\alpha)H^{\alpha+1}\Delta_\Sigma H\\
&&-H^{\alpha}\left[\frac{1}{2H^{2\alpha}|A|^2} |\nabla_\Sigma H|
^2\left(
2H^{\alpha+1}+\alpha \frac v{H}\right)^2-2|A|^4
-2|A|\|D^2 H\|
+2 H^2(|A|^2-K)\right]\\
&&
-\frac{2\alpha} H|\nabla_\Sigma H|^2\left(
2H^{\alpha+1}+\alpha \frac v{H}\right)\\
&&
-|A|^2\left(\alpha(\alpha-1)H^{\alpha-2}|\nabla_\Sigma H|^2+\alpha
H^{\alpha-1}\Delta H\right)
.
\end{eqnarray*}
Since $K<0$ at $x_0$ we can cancel it from the equation and then the
resulted inequality can be written in the following form
\[
0\le J_1+J_2 +J_3,
\]
where
\begin{eqnarray*}
J_1&=& (2+\alpha)H^{\alpha+1}\Delta_\Sigma H+2H^\alpha|A|\|D^2 H\|-
\alpha |A|^2H^{\alpha-1}\Delta H,\\
J_2&=&
(2+\alpha)(1+\alpha)H^\alpha|\nabla_\Sigma H|^2
-\frac{H^\alpha}{2H^{2\alpha}|A|^2} |\nabla_\Sigma H|^2\alpha^2
\frac {v^2}{H^2}
\\
&&
- \frac{2\alpha^2} H|\nabla_\Sigma H|^2\frac v{H}
-\alpha(\alpha-1)|A|^2H^{\alpha-2}|\nabla_\Sigma H|^2,
\\
J_3&=&
-\frac{4\alpha} H|\nabla_\Sigma H|^2
H^{\alpha+1}
-\frac{H^\alpha}{2H^{2\alpha}|A|^2} |\nabla_\Sigma H|^2
\left(
4H^{2\alpha+2}+4H^{\alpha+1}\alpha \frac v{H}\right)\\
&&
+2H^\alpha(|A|^4-H^2|A|^2).
\end{eqnarray*}
Putting $\omega=\frac{|A|^2}{H^2}$ for convenience, with $\omega=
1+\epsilon$ via selecting another $\epsilon$, we see that
\begin{eqnarray*}
J_1=H^{\alpha+1}
\left[
(2+\alpha)+\sqrt \omega-\alpha\omega
\right]O(\|D^2 H\|).
\end{eqnarray*}
Similarly,
\begin{eqnarray*}
J_2
&=&
H^{\alpha}|\nabla_\Sigma H|^2
\left[
\alpha^2+3\alpha+2
-
\alpha^2\frac1{2\omega}(1-\omega)^2
+
2\alpha^2|1-\omega|
-
\alpha(\alpha-1)\omega
\right]\\
&=&
H^{\alpha}|\nabla_\Sigma H|^2
\left[\alpha^2+3\alpha+2-\alpha^2 \frac{\epsilon^2}{2(1+\epsilon)} +2\alpha^2 \epsilon
-
(\alpha^2-\alpha)(1+\epsilon)
\right]
\\
\end{eqnarray*}
where we used $v=H^{\alpha+2}(1-\omega)$ and that $\alpha$ is
negative.
Finally,
\begin{eqnarray*}
J_3&=&
-\frac{4\alpha} H|\nabla_\Sigma H|^2
H^{\alpha+1}
-\frac{H^\alpha}{2H^{2\alpha}|A|^2} |\nabla_\Sigma H|^2
\left(
4H^{2\alpha+2}+4H^{\alpha+1}\alpha \frac v{H}\right)\\
&&
+2H^\alpha(|A|^4-H^2|A|^2)\\
&=&
-4\alpha H^\alpha|\nabla_\Sigma H|^2
-2H^\alpha\frac{|\nabla_\Sigma H|^2}\omega
+
2\alpha H^\alpha\frac{|\nabla_\Sigma H|^2}\omega |1-\omega|\\
&&
+2H^{\alpha+4}(\omega^2-\omega)\\
&=&
H^\alpha|\nabla_\Sigma H|^2
\left[
-4\alpha
-\frac{2}\omega
+
2\alpha \frac{|1-\omega|}\omega
\right]+2H^{\alpha+4}(\omega^2-\omega)\\
&=&
H^\alpha|\nabla_\Sigma H|^2
\left[
-4\alpha
-\frac{2}{1+\epsilon}
+
2\alpha \frac{\epsilon}{1+\epsilon}
\right]\\
&&
+2H^{\alpha+4}(\omega^2-\omega).
\end{eqnarray*}
By assumption $\omega$ is bounded, hence for the last term we have
$2H^{\alpha+4}(\omega^2-\omega)=O(H^{\alpha+4})$, and consequently
\[
J_2+J_3=-H^\alpha|\nabla_\Sigma H|^2 q(\epsilon, \alpha)+O(H^{\alpha+4}),
\]
where
\begin{eqnarray*}
q(\epsilon, \alpha)
&=&
-\left[\alpha^2+3\alpha+2-\alpha^2 \frac{\epsilon^2}{2(1+\epsilon)} +2\alpha^2
\epsilon
-
(\alpha^2-\alpha)(1+\epsilon)
\right]\\
&&
-\left[
-4\alpha
-\frac{2}{1+\epsilon}
+
2\alpha \frac{\epsilon}{1+\epsilon}
\right]\\
&=&
-\left[2-\alpha^2 \frac{\epsilon^2}{2(1+\epsilon)} +\alpha^2 \epsilon
+\alpha\epsilon
\right]\\
&&
+\frac{2-2\alpha\epsilon}{1+\epsilon}
\\
&=&
\frac1{2(1+\epsilon)}
\left[
-4-4\epsilon
+\alpha^2 \epsilon^2
-2\alpha^2 \epsilon-2\alpha^2 \epsilon^2
-2\alpha\epsilon-2\alpha\epsilon^2
+4-4\alpha\epsilon
\right]
\\
&=&
\frac\epsilon{2(1+\epsilon)}
\left[
-4
-\alpha^2 \epsilon
-2\alpha^2
-6\alpha-2\alpha\epsilon
\right].
\\
\end{eqnarray*}
Combining, and noting that
mean curvature is
$$
H=\mu-g,
$$
and in particular,
$$
|\nabla_\Sigma H|= |\nabla_\Sigma g|,
$$
we see that
\begin{eqnarray*}
0\le J_1+J_2+J_3
&\le&
H^{\alpha+1}
\left[
(2+\alpha)+\sqrt \omega-\alpha\omega
\right]O(\|D^2 H\|)\\
&&
-H^\alpha|\nabla_\Sigma H|^2 q(\epsilon, \alpha)+O(H^{\alpha+4})\\
&\le&
H^{\alpha+1}\left(O(1)-\frac{\sigma(H)} Hc_0+O(H^3)\right)<0
\end{eqnarray*}
provided $\lim_{t\to 0^+}\frac{\sigma(t)}t=\infty$ and $c_0>0$.
Observe that
choosing, say, $\alpha=-1$,
one is able to choose $c_0$ with
\begin{eqnarray*}
q(\epsilon, -1)&=&
\frac\epsilon{2(1+\epsilon)}
\left[
-4
-\epsilon
-2
+6+2\epsilon
\right]\\
&=&
\frac{\epsilon^2}{2(1+\epsilon)}=c_0.
\\
\end{eqnarray*}
\end{proof}
\begin{rem}
Recalling that the mean curvature is
$$
H=\mu-g
$$
and
supposing
$g(x)=g(|x|)$ is non-decreasing, Theorem \ref{@'} yields:
$$
a=\Big[ \frac{3m}{4\pi} \Big]^{\frac{1}{3}},
$$
$$
\partial E_m=\Sigma,
$$
$$
|E_m|=m;
$$
in particular, note that thanks to
$$
|A|^2 =\frac{2}{a^2}
$$
$$
H=\frac{2}{a},
$$
when $g_t$ is a perturbation of $g$ so that $|\nabla_{\Sigma_t} g_t|>0$, $\Sigma_t$
is a variation of the sphere
if $m > m_1$ where $m_1>0$ is a computable constant. In particular,
a continuity argument then generates convexity for perturbations
(which in the general case may not be radial/convex). If $m$ is
large, supposing the potential has convex sub-level sets, convexity
is known; nevertheless the previous implies an explicit bound and
also convexity for possible potentials with general sub-level sets.
\end{rem}
\begin{rem}
For general $f$, we have $H_f: =\langle D^2f, A\rangle=\mu-g$ (supposing a system where $A$ is diagonal). Assuming small
$H_f>0, $ we can extend the previous estimate to non-constant $f$, provided that
$D^2H$ remains bounded (for estimating $J_1$),
$H$ is small and positive.
\end{rem}
{\bf Acknowledgement.}  Some of the research was completed at the University of Edinburgh while the first author was a visiting scholar. The support is kindly acknowledged. 
\bibliographystyle{amsalpha}
\bibliography{References}
\end{document}